\documentclass{amsart}
\usepackage[utf8]{inputenc}
\newtheorem{theorem}{Theorem}[section]

\newtheorem{lemma}{Lemma}[section]

\newtheorem{corollary}{Corollary}[section]

\newtheorem{remark}{Remark}
\title[Blow--up Navier--Stokes revisited]{Lower bounds for possible singular solutions for the Navier--Stokes and Euler equations revisited.}
\author{Jean C. Cortissoz and Julio A. Montero}
\address{Departamento de Matem\'aticas\\ Universidad de los Andes\\ Bogot\'a DC, COLOMBIA}
\begin{document}
\begin{abstract}
In this paper we give optimal lower bounds for the blow-up rate of the 
 $\dot{H}^{s}\left(\mathbb{T}^3\right)$-norm, $\frac{1}{2}<s<\frac{5}{2}$, of a putative singular solution
of the Navier-Stokes equations, and we also present an elementary proof for a lower bound on blow-up rate
of the Sobolev norms of possible singular solutions to the 
Euler equations when $s>\frac{5}{2}$. 
\end{abstract}
\maketitle
\section{Introduction}

The Navier-Stokes equations for an incompressible viscous fluid of viscosity $\nu=1$ are given by
\begin{equation*}
\tag{NS}
\label{Navierstokes}
\left\{
\begin{array}{l}
u_t-\Delta u+u\cdot\nabla u+\nabla p=0 \quad \mbox{in} \quad X\times \left(0,\infty\right)\\
u\left(x,0\right)=\psi, \quad div \,u=0.
\end{array}
\right.
\end{equation*} 
where, in this paper, $X=\mathbb{R}^3$ or $X=\mathbb{T}^3$.
It is known that given an initial condition of finite energy and which belongs to $\dot{H}^s\left(X\right)$
there is an interval of time
$\left(0,\eta\right)$, $\eta>0$, for which there is a unique smooth solution to (\ref{Navierstokes}) in
$C\left(\left[0,T\right];\dot{H}^s\left(X\right)\right)$. Let then $T>0$ be the largest
$\eta>0$ for which the unique solution with initial data $\psi\in \dot{H}^s\left(X\right)$ remains smooth. It is unknown whether $T<\infty$ or 
$T=\infty$. In the case that $T<\infty$, there is the interesting question of estimating a rate at which 
the $\dot{H}^s$-norm blows-up.
In  \cite{CortissozMonteroPinilla}
the authors, based on ideas presented by Robinson, Sadowski and Silva in \cite{RobinsonSadowskiSilva}, showed an almost optimal
 lower bound for the blow--up rate of solutions of the Navier--Stokes equations with periodic boundary conditions 
on a bounded maximal interval of existence $\left(0,T\right)$, $T<\infty$, 
when this solution belongs
 to $\dot{H}^{\frac{3}{2}}(\mathbb{T}^3)\cap\dot{H}^{\frac{5}{2}}(\mathbb{T}^3)$. 
To be more precise, it was shown that a regular solution of the Navier--Stokes equation 
whose maximal interval of existence (or regularity) is $\left(0,T\right)$, must satisfy
\[
\left\|u\left(t\right)\right\|_{H^{\frac{3}{2}}\left(\mathbb{T}^3\right)}\geq \frac{c}{\sqrt{\left(T-t\right)\left|\log\left(T-t\right)\right|}},
\]
for a constant $c>0$. In this paper we go a little further and give a proof of the 
expected optimal lower blow--up rate. Namely, we prove the following the following estimate
on the blow--up rate of putative singular solutions to the Navier--Stokes equations: 
\begin{equation*}
  \frac{C}{t^\frac{1}{2}}\leq \left\|u(T-t)\right\|_{\dot{H}^{\frac{3}{2}}(\mathbb{T}^3)}, \quad C>0.
\end{equation*}
The proof of this result requires a detailed inspection of the bounds on the nonlinear term
 of the Navier--Stokes equations found in \cite{RobinsonSadowskiSilva},
and the application of an interpolation technique inspired by the method used by Hardy to prove Carlson's inequality (see \cite{LarsonPecaric}).
We must add that this problem using different techniques has been treated in the papers \cite{CheskidovZaya}
and \cite{Robinsonetal}.

The lower blow--up rates for putative singular solutions to the Navier--Stokes equations can be interpreted as a regularity criterion 
for solutions of the equation (as they give a lower bound on the size of the maximal interval of existence).
These blow--up estimates were first stated for the $L^p$ spaces, $p > 3$, 
without proof by Leray in his remarkable paper \cite{Leray}, and proved by Giga in \cite{Giga} via semigroup theory. 
In this paper, we rather follow the elementary, and improve on, the proof on homogeneous Sobolev spaces given by Robinson, Sadowski and Silva for their 
blow--up estimates.

On the other hand, there exists the related problem of investigating the possible blow-up behavior of solutions to the incompressible Euler equations:
\begin{equation*}
\tag{E}
\label{Euler}
\left\{
\begin{array}{l}
u_t+u\cdot\nabla u+\nabla p=0 \quad \mbox{in} \quad X\times \left(0,\infty\right)\\
u\left(x,0\right)=\psi, \quad \mbox{div} \!\left(u\right)=0.
\end{array}
\right.
\end{equation*} 
In fact, recently in a very nice paper \cite{ChenPavlovic}, Chen and Pavlović showed the following (although they state their result
in $\mathbb{R}^3$ and we do in $\mathbb{T}^3$, our arguments apply in both cases, see Remark 1).
\begin{theorem}
\label{th:ChenPavlovic}
Let $u\left(x,t\right)$ be a solution of the periodic Euler equations in the class
\begin{equation}
\label{class}
C^1\left(\left[0,T\right],\dot{H}^{\frac{3}{2}+\delta}\left(\mathbb{T}^3\right)\right)\cap
C\left(\left[0,T\right],\dot{H}^{\frac{5}{2}+\delta}\left(\mathbb{T}^3\right)\right),\quad \delta>0,
\end{equation}
and let $T>0$ be the minimum time for which $u$ cannot be continued in the class (\ref{class}). Then, there exists a finite, positive constant 
$C\left(\delta,\left\|u\left(0\right)\right\|_{L^2\left(\mathbb{T}^3\right)}\right)$ such that
\[
\left\|u\left(t\right)\right\|_{\dot{H}\left(\mathbb{T}^3\right)^{\frac{5}{2}+\delta}}\geq C\left(\delta,\left\|u\left(0\right)\right\|_{L^2\left(\mathbb{T}^3\right)}\right)
\left(\frac{1}{T-t}\right)^{2+\frac{2}{5}\delta}.
\]
\end{theorem}

The proof of this theorem given in \cite{ChenPavlovic} relies on obtaining a single exponential bound on the $H^s$ norms of a solution
of the Euler equations via a lenght parameter introduced by P. Constantin in \cite{Constantin}. In this paper we follow 
the approach suggested in \cite{RobinsonSadowskiSilva} in conjunction with some ideas
presented in \cite{CortissozMonteroPinilla}, to give a less involved proof of Theorem \ref{th:ChenPavlovic}.
 
Part of this paper was written while the second author was visiting the Mathematics Department at Cornell University, and he is quite
grateful for their warm hospitality -and in particular to Prof. Tim Healy for
his encouragement. He also must acknowledge the support of Colciencias and his home institution,
the Universidad de los Andes for making this visit possible, and his advisor (the first named author of this paper)
for his encouragement, and his almost always insightful observations. The first author wants to thank the second author for
being a great student and colleague, and for all these wonderful years of shared mathematical enthusiasm. 
He also wants to thank his home institution, the Universidad de los Andes, for providing an excellent research environment and
economic support (Proyecto Semilla P15.160322.009).

\section{The blow up rate for the Navier-Stokes equations}

The next statement is essentially the same given in \cite{RobinsonSadowskiSilva}. 
The main difference is that we show a proof which includes the case when the solution belongs 
to $\dot{H}^{\frac{3}{2}}(\mathbb{T}^3)\cap  \dot{H}^{\frac{5}{2}}(\mathbb{T}^3)$. From now on, in this
paper we shall use the notation
\[
\left\|u\right\|_s:=\left\|u\right\|_{\dot{H}^s\left(\mathbb{T}^3\right)},
\]
and $\hat{u}_k$ refers to the Fourier wavenumber of wavevector $k$ of the function $u$.
\begin{theorem}
  \label{th:Blow_up}
Let $u(x,t) = (u_1,u_2,u_3)$ be a solution Navier--Stokes equations whose maximum interval of existence is
$\left(0,T\right)$, $0<T<\infty$, and such that 
$u\in  C((0,T),\dot{H}^s(\mathbb{T}^3)\cap  \dot{H}^{s+1}(\mathbb{T}^3))$, with $\displaystyle \frac{1}{2}<s<\frac{5}{2}$. 
Then the following estimate holds
\begin{equation}
  \label{eq:BlowUpRate}
  \frac{C_s}{t^{\frac{1}{2}(s-\frac{1}{2})}} \leq \left\|u(T-s)\right\|_{\dot{H}^s(\mathbb{T}^3)}.
\end{equation} 
\end{theorem}

\begin{proof}
  First, we must recall the energy inequality found in \cite{RobinsonSadowskiSilva}: 
\begin{equation}
  \label{eq:EnergyInequality}
  \frac{1}{2}\frac{d}{dt}(\left\|u(t)\right\|_s^2) + 4\pi^2\left\|u(t)\right\|_{s+1}\leq C_s
  \left(
\sum_k \left|\widehat{u}_k\right|\left|k\right|^r
  \right)\left\|u(t)\right\|_s\left\|u\right\|_{s+1-r},
\end{equation}
with $0\leq r\leq 1$. For the sake of completeness we will give a proof of this inequality below.

Now we pick $\displaystyle r=\frac{1}{2}\left(s-\frac{1}{2}\right)$, 
and apply the interpolation technique employed by Hardy in his proof of Carlson's inequality (see \cite{LarsonPecaric}), 
to the first factor on the right hand side of (\ref{eq:EnergyInequality}), to obtain:

\begin{align*}
  \sum_k \left|\widehat{u}_k\right|\left|k\right|^{\frac{1}{2}\left(s-\frac{1}{2}\right)}=& 
\sum_k \left|\widehat{u}_k\right|
\left|k\right|^{\frac{1}{2}\left(s-\frac{1}{2}\right)}\frac{\sqrt{a|k|^{s+\frac{1}{2}}+b|k|^{s+\frac{5}{2}}}}{\sqrt{a|k|^{s+\frac{1}{2}}+b|k|^{s+\frac{5}{2}}}}\\
\leq& \left(a\left\|u\right\|^2_s+b\left\|u\right\|^2_{s+1}\right)^{\frac{1}{2}}
\left(
\sum_k \frac{1}{a\left|k\right|^{s+\frac{1}{2}}+b\left|k\right|^{s+\frac{5}{2}}}
\right)^{\frac{1}{2}}\\
\leq& \left(a\left\|u\right\|^2_s+b\left\|u\right\|^2_{s+1}\right)^{\frac{1}{2}}
\left(
\frac{4\pi}{\sqrt{ab}}
\left(
\frac{\sqrt{a}}{\sqrt{b}}
\right)^{\frac{3}{2}-s}\int_0^\infty \frac{y^{\frac{3}{2}-s}}{1+y^2}dy
\right)^{\frac{1}{2}},
\end{align*}
if we choose $a=\left\|u(t)\right\|^2_{s+1}$ and $b=\left\|u(t)\right\|^2_{s}$ then the energy inequality (\ref{eq:EnergyInequality})
becomes
\begin{equation}
\label{eq:EnergyInequality2}
  \frac{1}{2}\frac{d}{dt}(\left\|u(t)\right\|_s^2) + 4\pi^2 \left\|u(t)\right\|^2_{s+1}
\leq C_s \left\|u(t)\right\|_s^{\frac{s}{2}+\frac{3}{4}}\left\|u(t)\right\|_{s+1}^{\frac{5}{4}-\frac{s}{2}}\left\|u(t)\right\|_{\frac{s}{2}+\frac{5}{4}}.
\end{equation}
Now, observe that $\displaystyle \frac{s}{2}+\frac{5}{4}=
\left(
  \frac{s}{2}-\frac{1}{4}
\right)s +
\left(
\frac{5}{4}-\frac{s}{2}
\right)(s+1)$, so by interpolation between homogeneous Sobolev spaces, we get
\begin{equation*}
  \left\|u\right\|_{\frac{s}{2}+\frac{5}{4}}\leq \left\|u\right\|_s^{\frac{s}{2}-\frac{1}{4}}\left\|u\right\|_{s+1}^{\frac{5}{4}-\frac{s}{2}}.
\end{equation*}
Therefore, from inequality (\ref{eq:EnergyInequality2}) we obtain
\begin{equation*}
 \frac{1}{2} \frac{d}{dt}(\left\|u(t)\right\|_s^2)+ 4\pi^2 \left\|u(t)\right\|^2_{s+1}  \leq C_s \left\|u(t)\right\|_s^{s+\frac{1}{2}}
\left\|u(t)\right\|_{s+1}^{\frac{5}{2}-s}.
\end{equation*}
It is time to use Young's inequality $\displaystyle ab\leq \frac{a^p}{p}+\frac{b^q}{q}$, $\displaystyle \frac{1}{p}+\frac{1}{q}=1$,
with the choice $p=\frac{2\left(s+\frac{1}{2}\right)}{s-\frac{1}{2}}$ and $q=\frac{2}{\frac{5}{2}-s}$. We thus get
\[
\frac{1}{2} \frac{d}{dt}(\left\|u(t)\right\|_s^2)\leq c_s\left(\left\|u(t)\right\|_s^2\right)^{\left(1+\frac{1}{s-\frac{1}{2}}\right)}.
\] 
Finally, by integrating between $T-t$ and $T$ the previous estimate, inequality (\ref{eq:BlowUpRate}) follows.
\end{proof}

\begin{remark}
Theorem \ref{th:Blow_up} is also valid when we consider the case of the whole space, i.e., for solutions
 $u(t) \in \dot{H}^{\frac{3}{2}}(\mathbb{R}^3)\cap \dot{H}^{\frac{5}{2}}(\mathbb{R}^3)$, this 
because all the calculations leading to its proof are valid on $\mathbb{R}^3$ if we change sums by integrals.  
\end{remark}

As promised, we give a proof of inequality (\ref{eq:EnergyInequality}). 
It is a consequence of the following lemma 
(see the proof of Lemma 3.1 in \cite{RobinsonSadowskiSilva})
which gives an estimate
of the nonlinear term
\[
\left|\left(B\left(u,u\right),u\right)_{\dot{H}^{s}}\right|,
\]
where
\[
B\left(u,u\right)=P\left(u\cdot\nabla u\right),
\]
and $P$ is the Leray projector. 
\begin{lemma}
  \label{th:InterpLog}
  For any \(s>1\) and \(0\leq r\leq 1\), we have
  \begin{equation*}
    \left|\sum_{k\in\dot{\mathbb{Z}^3}}\sum_{q\in\dot{\mathbb{Z}^3}}|k|^{2s}(k\cdot \widehat{u}_{k-q})(\widehat{u}_q\cdot\overline{\widehat{u}}_k)\right|
    \leq c_s \left(\sum_{k\in\dot{\mathbb{Z}^3}} |k|^r|\widehat{u}_k|\right)\left\|u\right\|_{s}\left\|u\right\|_{s+1-r},
  \end{equation*}
  for all \( u\in \dot{H}^{s+1-r}(\mathbb{T}^3)\cap F^r\). Here 
  $\hat{u}_k$ denotes the Fourier wavenumber of $u$
  with wavevector $k$, and $\overline{z}$ 
  denotes the complex conjugate of $z$.
\end{lemma}
\begin{proof}
  Since \(P\)
  is self-adjoint and \(u(x,t)\)
  is divergence free , we have that (see \cite{ConstantinFoias},
  chapter 6 p. 53)
  \begin{equation*}
    \sum_{k\in \dot{\mathbb{Z}}^3}\sum_{q\in \dot{\mathbb{Z}}^3}|q|^s|k|^s(\widehat{u}_{k-q}\cdot q)(\widehat{u}_q\cdot \overline{\widehat{u}}_k) = 0.
  \end{equation*}
  Using the inequality ( see \cite{HardyLittlewoodPoyla}, p.39)
  \begin{equation*}
    \left|
      |x|^s-|y|^s
    \right|\leq s(2^s)|x-y|
    \left(
      |x-y|^{s-1}+|y|^{s-1}
    \right),\quad s>1;
  \end{equation*}
  and the well know inequality
  \begin{equation*}
    (|x|+|y|)^r \leq |x|^r+|y|^r \textrm{ if } 0\leq r \leq 1,
  \end{equation*}
  we obtain:
  \begin{align*}
    &\left|\sum_{k\in\dot{\mathbb{Z}^3}}\sum_{q\in\dot{\mathbb{Z}^3}}|k|^{2s}(k\cdot \widehat{u}_{k-q})(\widehat{u}_q\cdot\overline{\widehat{u}}_k)\right|\\
    &\quad\quad
      =\left|\sum_{k\in\dot{\mathbb{Z}^3}}\sum_{q\in\dot{\mathbb{Z}^3}}|k|^{2s}(k\cdot
      \widehat{u}_{k-q})(\widehat{u}_q\cdot\overline{\widehat{u}}_k)
      -\sum_{k\in \dot{\mathbb{Z}}^3}\sum_{q\in
      \dot{\mathbb{Z}}^3}|q|^s|k|^s(\widehat{u}_{k-q}\cdot
      q)(\widehat{u}_q\cdot \overline{\widehat{u}}_k)
      \right|\\
    &\quad\quad \leq  \sum_{k\in\dot{\mathbb{Z}^3}}\sum_{q\in\dot{\mathbb{Z}^3}}|k|^{s}|(q\cdot\widehat{u}_{k-q})|\,|\widehat{u}_q|\,|\widehat{u}_k|\,||k|^s- |q|^s|\\    
    &\quad\quad\leq  s2^{s-1} \sum_{k\in\dot{\mathbb{Z}^3}}\sum_{q\in\dot{\mathbb{Z}^3}}|k|^{s}|(q\cdot\widehat{u}_{k-q})||\widehat{u}_q||\widehat{u}_k| |k-q|(|k-q|^{s-1}+|q|^{s-1})\\
    &\quad\quad\leq  s2^s \sum_{q\in\dot{\mathbb{Z}^3}}\sum_{k\in\dot{\mathbb{Z}^3}}|k|^{s}|k-q|^s|q||\widehat{u}_{k-q}||\widehat{u}_q||\widehat{u}_k|\\
    &\quad\quad\leq  s2^s \sum_{q\in\dot{\mathbb{Z}^3}}\sum_{k\in\dot{\mathbb{Z}^3}}|k|^{s}|k-q|^s|q|^r|q|^{1-r}|\widehat{u}_{k-q}||\widehat{u}_q||\widehat{u}_k|\\
    &\quad\quad\leq  s2^s \sum_{q\in\dot{\mathbb{Z}^3}}\sum_{k\in\dot{\mathbb{Z}^3}}|k|^{s}|k-q|^s|q|^r(|k-q|^{1-r}+|k|^{1-r})|\widehat{u}_{k-q}||\widehat{u}_q||\widehat{u}_k|\\
    &\quad\quad\leq  s2^{s+1} \sum_{q\in\dot{\mathbb{Z}^3}} |q|^r|\widehat{u}_q|\sum_{k\in\dot{\mathbb{Z}^3}}|k-q|^{s}|\widehat{u}_{k-q}||k|^{s+1-r}|\widehat{u}_k|\\
    &\quad\quad\leq
      s2^{s+1}\left(\sum_{k\in\dot{\mathbb{Z}^3}}|k|^r|\widehat{u}_k|\right)\left\|u\right\|_{s}\left\|u\right\|_{s+1-r},
  \end{align*}
  which is what we wanted to prove.
\end{proof}

The previous proof gives us also an lower bound on size on the maximal interval of existence. Indeed,
the following result holds.

\begin{corollary}
  Let $u(x,t)$ be a solution of the Navier--Stokes equations with initial condition $u_0(x)\in \dot{H}^s(\mathbb{T}^3)$, $\frac{1}{2}<s<\frac{5}{2}$, 
and let $T>0$ be the minimum time for blow--up. Then

  \begin{equation}
    \label{eq:Time}
    \frac{K_s}{\left(\left\|u_0\right\|_s\right)^{\frac{4}{2s-1}}}\leq T.
  \end{equation}

\end{corollary}

\section{The blow up rate for the Euler equations}

Estimate (\ref{eq:EnergyInequality}) can be used to provide the promised elementary proof of Theorem \ref{th:ChenPavlovic}. But before we present our
proof we will need the following estimate.
\begin{lemma}
Let
\[
\left\|u\right\|_{F^1}=\sum_{k\in \mathbb{Z}^3}\left|k\right|\left|\hat{u}_k\right|.
\]
There is a constant $c>0$ which only depends on $s$ such that
\[
\left\|u\right\|_{F^1}\leq c\left\|u\right\|_{L^2\left(\mathbb{T}^3\right)}^{\frac{s-\frac{5}{2}}{s}}\left\|u\right\|_{s}^{\frac{5}{2s}}.
\]
\end{lemma}
\begin{proof}
We have that
\begin{eqnarray*}
\sum_k\left|\hat{u}_k\right|\left|k\right|&=&\sum_k\left|\hat{u}_k\right|\left|k\right|\frac{\sqrt{a+b\left|k\right|^{2s-2}}}{\sqrt{a+b\left|k\right|^{2s-2}}}\\
&\leq& \left(\sum_k\left|\hat{u}_k\right|^2\left|k\right|^2\left(a+b\left|k\right|^{2s-2}\right)\right)^{\frac{1}{2}}
\left(\sum_k\frac{1}{a+b\left|k\right|^{2s-2}}\right)^{\frac{1}{2}}\\
&\leq&c_s\left(a\left\|u\right\|^2_1+b\left\|u\right\|_s^2\right)\left(\int_0^{\infty}\frac{x^2dx}{a+bx^{2s-2}}\right)^{\frac{1}{2}}\\
&\leq&c_s\left(a\left\|u\right\|^2_1+b\left\|u\right\|_s^2\right)\frac{1}{\sqrt{a}}\left(\frac{a}{b}\right)^{\frac{3}{2\left(2s-2\right)}}.
\end{eqnarray*}
We let $a=\left\|u\right\|^2_s$ and $b=\left\|u\right\|_{1}^2$ to obtain (for a new constant $c_s>0$)
\begin{eqnarray*}
\left\|u\right\|_{F^1}&\leq& c_s\left\|u\right\|_1\left(\frac{\left\|u\right\|_s}{\left\|u\right\|_1}\right)^{\frac{3}{2s-2}}\\
&=&c_s\left\|u\right\|_1^{\frac{2s-5}{2s-2}}\left\|u\right\|^{\frac{3}{2s-2}}_s.
\end{eqnarray*}
Now we use the Sobolev interpolation inequality
\[
\left\|u\right\|_1\leq \left\|u\right\|_{L^2\left(\mathbb{T}^3\right)}^{\frac{s-1}{s}}\left\|u\right\|^{\frac{1}{s}}_s,
\]
to obtain
\begin{eqnarray*}
\left\|u\right\|_{F^1}&\leq& c_s\left\|u\right\|_{L^2\left(\mathbb{T}^3\right)}^{\frac{2s-5}{2s}}\left\|u\right\|^{\frac{1}{s}\frac{2s-5}{2s-2}+\frac{3}{2s-2}}_s\\
&=&c_s\left\|u\right\|_{L^2\left(\mathbb{T}^3\right)}^{\frac{2s-5}{2s}}\left\|u\right\|^{\frac{5}{2s}}_s.
\end{eqnarray*}

\end{proof}

We are ready to prove the estimate of Chen and Pavlović. Indeed, proceeding as we did in the proof
of Theorem \ref{th:Blow_up}, for the Euler equations we obtain (i.e.,
by inequality (\ref{eq:EnergyInequality}) with $r=1$;
the extra positive term on the left-hand side does not appear, due to the lack of the diffusion term in the Euler equations) : 
\[
\frac{d}{dt}\left\|u\right\|_s^2\leq c_s\left\|u\right\|_{F^1}\left\|u\right\|_s^2,
\]
(this is just equation (6.2) in \cite{RobinsonSadowskiSilva}), 
and hence by the previous lemma we arrive at the differential inequality
\begin{eqnarray*}
\frac{d}{dt}\left\|u\right\|_s^2&\leq& c_s\left\|u\right\|_{L^2\left(\mathbb{T}^3\right)}^{\frac{2s-5}{2s}}\left\|u\right\|^{\frac{5}{2s}+2}_s.
\end{eqnarray*}
For a regular solution to Euler equation, it is well-known that for any $t\geq 0$,
\[
\left\|u\left(t\right)\right\|_{L^2\left(\mathbb{T}^3\right)}\leq \left\|u\left(0\right)\right\|_{L^2\left(\mathbb{T}^3\right)},
\]
so we obtain an inequality (here the constant involved depends on $\left\|u\left(0\right)\right\|_{L^2\left(\mathbb{T}^3\right)}$)
\[
\frac{d}{dt}\left\|u\right\|_s^2\leq c\left(\left\|u\left(0\right)\right\|_{L^2\left(\mathbb{T}^3\right)},s\right)\left\|u\right\|^{\frac{5}{2s}+2}_s.
\]
Let $s=\frac{5}{2}+\delta$. Then our inequality becomes
\[
\frac{d}{dt}\left\|u\right\|_{\frac{5}{2}+\delta}^2\leq c\left(\left\|u\left(0\right)\right\|_{L^2\left(\mathbb{T}^3\right)},\delta\right)
\left\|u\right\|^{2+\frac{1}{1+\frac{2}{5}\delta}}_{\frac{5}{2}+\delta}.
\]
Integrating the previous differential inequality from $t$ to $T$ (and assuming blow up at $T$) we get
\[
\frac{1}{\left\|u\left(t\right)\right\|^\frac{1}{1+\frac{2}{5}\delta}_{\frac{5}{2}+\delta}}\leq c\left(T-t\right),
\]
where $c$ is a constant that only depends on $\left\|u\left(0\right)\right\|_{L^2\left(\mathbb{T}^3\right)}$ and $\delta$. Solving for
$\left\|u\left(t\right)\right\|_{\frac{5}{2}+\delta}$ finishes the proof.

\vspace{.2in}
\noindent
{\bf Remark 2.} As commented before in the case of the Navier-Stokes equations, in the proofs in this section
it is possible to replace $\mathbb{T}^3$ by $\mathbb{R}^3$.

\section{Final comments: Some open questions}
Theorem \ref{th:Blow_up} includes the optimal lower bound for blow--up rates 
when $u\in \dot{H}^{\frac{3}{2}}(\mathbb{T}^3)\cap\dot{H}^{\frac{5}{2}}(\mathbb{T}^3)$; 
this particular case was missing in the proof given in \cite{RobinsonSadowskiSilva}, and in \cite{CortissozMonteroPinilla}
a non optimal bound was proved. 
These bounds raise the following question: If there exists some $C>0$ such that $
\displaystyle \left\|u(T-t)\right\|_s\leq C t^{-\frac{1}{2}(s-\frac{1}{2})}$ ,  
does $\left\|u(T-t)\right\|_s$ blow--up? Furthermore, a lower blow-up rate for $u\in \dot{H}^\frac{1}{2}(\mathbb{T}^3)$, for putative
blow--up solutions to the Navier--Stokes equations, is yet unknown.

\end{document}